\documentclass{elsarticle}

\usepackage{amsmath,amssymb,amsthm,mathrsfs}
\usepackage{tikz}
\usetikzlibrary{positioning}
\usepackage{color}

\theoremstyle{plain}
\numberwithin{equation}{section}
\theoremstyle{plain}
\newtheorem{theorem}{Theorem}[section]
\newtheorem{corollary}[theorem]{Corollary}
\newtheorem{lemma}[theorem]{Lemma}
\newtheorem{proposition}[theorem]{Proposition}
\theoremstyle{definition}
\newtheorem{definition}[theorem]{Definition}

\newtheorem{remark}[theorem]{Remark}

\newcommand{\ands}{\quad\mbox{and}\quad}
\newcommand{\kernel}{\operatorname{Ker}}

\newcommand{\Dom}{\operatorname{Dom}}
\newcommand{\codim}{\operatorname{codim}}
\newcommand{\Index}{\operatorname{Index}}
\newcommand{\Rat}{{\mathrm{Rat}}}
\newcommand{\Ran}{\operatorname{Ran}}
\newcommand{\diag}{\operatorname{Diag}}

\newcommand{\cP}{{\mathcal P}}

\newcommand{\BC}{{\mathbb C}}
\newcommand{\BT}{{\mathbb T}}

\newcommand{\BD}{{\mathbb D}}

\newcommand{\BP}{{\mathbb P}}
\newcommand{\BZ}{{\mathbb Z}}
\newcommand{\wtil}[1]{{\widetilde{#1}}}

\newcommand{\de}{\delta}

\newcommand{\om}{\omega}\newcommand{\Om}{\Omega}

\newcommand{\ov}[1]{{\overline{#1}}}

\newcommand{\pmat}[1]{\ensuremath{\begin{pmatrix} #1 \end{pmatrix}}}

\newcommand{\sbpm}[1]{\left(\begin{smallmatrix} #1\end{smallmatrix}\right)}

\begin{document}

\begin{frontmatter}

\title{A Toeplitz-like operator with rational matrix symbol having poles on the unit circle:\ Fredholm characteristics}

\author[GJG]{G.J. Groenewald}\ead{Gilbert.Groenewald@nwu.ac.za}
\author[StH]{S. ter Horst}\ead{Sanne.TerHorst@nwu.ac.za}
\author[JJJ]{J.J. Jaftha}\ead{Jacob.Jaftha@uct.ac.za}
\author[AR]{A.C.M. Ran}\ead{a.c.m.ran@vu.nl}

\address[GJG]{School of Mathematical and Statistical Sciences,
North-West University,
Research Focus: Pure and Applied Analytics,
Private~Bag X6001,
Potchefstroom 2520,
South Africa.}

\address[StH]{School of Mathematical and Statistical Sciences,
North-West University,
Research Focus: Pure and Applied Analytics,
Private~Bag X6001,
Potchefstroom 2520,
South Africa
and
DSI-NRF Centre of Excellence in Mathematical and Statistical Sciences (CoE-MaSS),
Johannesburg,
South Africa}

\address[JJJ]{Numeracy Center, University of Cape Town, Rondebosch 7701; Cape Town; South Africa}

\address[AR]{Department of Mathematics, Faculty of Science, VU Amsterdam, De Boelelaan 1111, 1081 HV Amsterdam, The Netherlands and Research Focus: Pure and Applied Analytics, North-West~University, Potchefstroom, South Africa}


\begin{abstract}
In a recent paper (Groenewald et al.\ {\em Complex Anal.\ Oper.\ Theory} \textbf{15:1} (2021)) we considered an unbounded Toeplitz-like operator $T_\Om$ generated by a rational matrix function $\Om$ that has poles on the unit circle $\BT$ of the complex plane. A Wiener-Hopf type factorization was proved and this factorization was used to determine some Fredholm properties of the operator $T_\Om$, including the Fredholm index. Due to the lower triangular structure (rather than diagonal) of the middle term in the Wiener-Hopf type factorization and the lack of uniqueness, it is not straightforward to determine the dimension of the kernel of $T_\Om$ from this factorization, and hence of the co-kernel, even when $T_\Om$ is Fredholm. In the current paper we provide a formula for the dimension of the kernel of $T_\Om$ under an additional assumption on the Wiener-Hopf type factorization. In the case that $\Om$ is a $2 \times 2$ matrix function, a characterization of the kernel of the middle factor of the Wiener-Hopf type factorization is given and in many cases a formula for the dimension of the kernel is obtained. The characterization of the kernel of the middle factor for the $2 \times 2$ case is partially extended to the case of matrix functions of arbitrary size.
\end{abstract}

\begin{keyword}
Toeplitz operators, unbounded operators, Fredholm characteristics, Toeplitz kernels, rational matrix functions

\emph{AMS subject classifications:} {Primary 47B35, 47A53; Secondary 47A68}

\end{keyword}

\end{frontmatter}

{\date{}}

\begin{center}
\emph{Dedicated to Albrecht B\"{o}ttcher on the occasion of his seventieth  birthday.}
\end{center}


\section{Introduction}

The first work on unbounded Toeplitz operators goes back to the 1950s and 1960s \cite{HW50,H63,R69}. The topic gained more momentum through a paper of Sarason \cite{S08}, who connected it to truncated Toeplitz operators \cite{S07}. The recent paper \cite{BGS17} of B\"{o}ttcher, Garoni, and Serra-Capizzano gives insights in possible applications of Toeplitz operators with unbounded symbols. The current paper is a continuation of our earlier work on unbounded Toeplitz-like operators with rational symbols that are allowed to have poles on the unit circle. While the case of scalar rational symbols with poles on the unit circle has a fairly satisfactory theory by now (see \cite{GtHJR1, GtHJR2,GtHJR3}), the case of matrix valued rational symbols is only partly understood. In \cite{GtHJR21} we investigated the question when the unbounded Toeplitz-like operator is Fredholm. It transpired that the answer is, as one would expected, that Fredholmness corresponds to the rational matrix symbol having no zeroes on the unit circle. This is the same condition as in the scalar case, and the same condition as in the bounded case (i.e., when there are no poles of the rational matrix symbol on the unit circle). In addition, a formula for the Fredholm index was given in \cite{GtHJR21}, which was based on a Wiener-Hopf type factorization (see Theorem \ref{T:fact} below). For a more complete picture, one would also want to have formulas or the dimension of the kernel and codimension of the range of the Toeplitz-like operator in case the operator is Fredholm. However, it seems that these are not easily determined from the Wiener-Hopf type factorization.

The goal of the current paper is to remedy that situation, at least partly, in the following manner. We shall identify a condition on the factors in the Wiener-Hopf type factorization that allows us to provide formulas for the dimension of the kernel and the codimension of the range in case the Toeplitz-like operator is Fredholm. In addition, we analyse in detail the case of $2 \times 2$ matrix functions, in which case we characterize the kernel of the Toeplitz-like operator corresponding to the `middle term' in the Wiener-Hopf type factorization, leading to a formula of the dimension of the kernel of the original Toeplitz-like operator in many cases. Finally, this characterization is partially extended to the general case.

To define the Toeplitz-like operators and state some results, we introduce some notation. For positive integers $m$ and $n$, let $\Rat^{m\times n}$ denote the space of $m \times n$ rational matrix functions, abbreviated to $\Rat^{m}$ if $n=1$. We write $\Rat^{m\times  n}(\BT)$ for the functions in $\Rat^{m\times n}$ with poles only on $\BT$ and with $\Rat_0^{m\times  n}(\BT)$ we indicate the functions in $\Rat^{m\times n}(\BT)$ that are strictly proper, that is, whose limit at $\infty$ exists and is equal to the zero-matrix. Also here, if $n=1$ we write $\Rat^{m}(\BT)$ and $\Rat_0^{m}(\BT)$ instead of $\Rat^{m\times 1}(\BT)$ and $\Rat_0^{m\times 1}(\BT)$, respectively. In the scalar case, i.e., $m=n=1$, we simply write $\Rat$, $\Rat(\BT)$ and $\Rat_0(\BT)$, as was done in \cite{GtHJR1,GtHJR2,GtHJR3}.

A rational matrix function has a pole at $z_0$ if any of its entries has a pole at $z_0$. A zero of a square rational matrix function $\Omega(z)$ is a pole of its inverse $\Om^{-1}(z)=\Om(z)^{-1}$, see for example \cite{DS74}.
In the case of a square matrix polynomial $P(z)$, the zeroes coincide with the zeroes of the polynomial $\det P(z)$, but for square rational matrix functions, zeroes can also occur at points where the determinant is nonzero. In particular, zeroes may occur at the same points as poles.

The space of $m\times n$ matrix polynomials will be denoted by $\cP^{m\times n}$, abbreviated to $\cP^m$ if $n=1$ and to $\cP\/$ if $m=n=1$. For all positive integers $k$ we write $\cP^{m\times n}_k$, $\cP^{m}_k$ and $\cP_k$ for the polynomials in $\cP^{m\times n}$, $\cP^{m}$ and $\cP$ of degree at most $k$, respectively. For notational convenience, we set $\cP_k=\{0\}$ for $k<0$ and extend this notation to $\cP^{m\times n}$ and $\cP^{m}$ in the same way as above. Hence, we interpret the zero-polynomial as having degree $-\infty$. By $L_m^p$ and $H_m^p$ we shall mean Banach spaces consisting of vector-functions with $m$-components who are all in $L^p$ and $H^p$, respectively.

We can now define our Toeplitz-like operators. Let $\Om \in\Rat^{m\times  m}$ with possibly poles on $\BT$ and $\det\Om(z)\not \equiv 0$. Define
$T_\Om \left (H^p_m \rightarrow H^p_m \right ), 1 < p < \infty,$ by
\begin{equation}\label{TOm}
\begin{aligned}
\Dom(T_\Om)=\left\{\begin{array}{ll}
f\in H^p_m :&
\begin{array}{l}
\Om f = h + \eta \textrm{ where } h\in L^p_m(\BT),\\
\ands \eta \in\Rat_0^{m}(\BT) \\
\end{array}\\
\end{array}
\right \}, \\
T_\Om f = \BP h \textrm{ with } \BP \textrm{ the Riesz projection of } L^p_m(\BT) \textrm{ onto }H^p_m.
\end{aligned}
\end{equation}
By the Riesz projection, $\BP$, we mean the projection of $L^p_m$ onto $H^p_m$ due to M. Riesz, see for example pages 149--153 in \cite{H62}, in contrast to the Riesz projection in spectral operator theory, due to F. Riesz, see for example pages 9 - 13 in \cite{GGK90b}. For simplicity we will consider the square case only, but many of the results in this paper extend to the non-square case.

Recall that an $m\times m$ rational matrix function is called a \emph{plus function\/} if it has no poles in the closed unit disk $\ov{\BD}$ and it is said to be a \emph{minus function\/} if it has no poles outside the open unit disk $\BD$, with the point at infinity included.
In the case where $\Omega(z)$ has no poles on the unit circle, it is known that $T_\Om$ is Fredholm if and only if $\Omega(z)$ has no zeroes on the unit circle. In that case
the function $\Omega(z)$ can be factorized as
\[
\Omega(z)=\Omega_-(z)D(z)\Omega_+(z),
\]
where $\Omega_-$ and $\Omega_-^{-1}$ are minus functions, $\Omega_+$ and $\Omega_+^{-1}$ are plus functions and $D$ is a diagonal matrix with monomial entries on the diagonal:\ $D(z)={\diag\, }(z^{\kappa_j})_{j=1}^n$, where we may arrange the integers $\kappa_j$ such that they are in decreasing order: $\kappa_1\geq \kappa_2 \geq \cdots \geq \kappa_n$, see, e.g., \cite{CG}. This Wiener-Hopf factorization plays an important role in the study of the Toeplitz operator $T_\Om$ because of the following relation between the factorization of the rational matrix function and a factorization of the operator: $T_\Om=T_{\Om_-}T_D T_{\Om_+}$, and moreover, $T_{\Om_-}$ and $T_{\Om_+}$ are invertible, with inverses given by, respectively, $T^{-1}_{\Om_-}=T_{\Om_-^{-1}}$ and $T^{-1}_{\Om_+}=T_{\Om_+^{-1}}$. As a consequence, the Fredholm properties of $T_\Om$ and $T_D$ coincide. To be precise, the dimension of the kernel and the codimension of the range of $T_\Om$ can be read off from the
Wiener-Hopf indices $\kappa_1, \ldots , \kappa_n$ as follows:
\[
{\rm dim\ Ker\,}T_\Om=\sum_{\{j\mid \kappa_j <0\}}\kappa_j, \qquad
{\rm codim\ Ran\,}T_\Om=\sum_{\{j\mid \kappa_j >0\}}\kappa_j .
\]
Although the factors $\Om_+$ and $\Om_-$ in the Wiener-Hopf factorization are not uniquely determined by $\Om$, the middle factor $D$ is, at least in case the diagonal entries are ordered decreasingly, and hence so are the Wiener-Hopf indices.

In Theorems 1.1 and 1.2 of \cite{GtHJR21} we considered the case where $\Omega$ has poles on the unit circle. The following weaker version of the Wiener-Hopf factorization was proved.


\begin{theorem}\label{T:fact}
Let $\Om\in\Rat^{m\times m}$ with $\det \Om \not\equiv 0$. Then
\begin{equation}\label{WHfact1}
\Om(z) = z^{-k}\Om_-(z) \Om_0(z) P_0(z) \Om_+(z),
\end{equation}
for some $k\geq0$, $\Om_+,\Om_0,\Om_-\in\Rat^{m\times m}$ and $P_0\in\cP^{m\times m}$ such that
\begin{itemize}
  \item[(i)] $\Om_-$ and $\Om_-^{-1}$ are minus functions;
  \item[(ii)] $\Om_+$  and $\Om_+^{-1}$ are plus functions;
  \item[(iii)] $\Om_0= \diag_{j=1}^m (\phi_j)$ with each $\phi_j\in\Rat(\BT)$ having zeroes and poles only on $\BT$;
  \item[(iv)] $P_0$ is lower triangular with $\det (P_0(z)) = z^N$ for some integer $N \geq 0$.
\end{itemize}
\end{theorem}

In fact, Theorem 1.2 in \cite{GtHJR21}, gives more restrictions on the middle factor

\begin{equation}\label{MiddleFact}
\Xi(z):=z^{-k}\Om_0 (z)P_0(z),
\end{equation}
making it possible to have $\Omega_0$ uniquely determined. Note that the diagonal entries of $P_0$ have to be of the form $z^{k_j}$, $j=1, \ldots , n$ for some integers $k_j$.
It was also shown in examples in \cite{GtHJR21}, Section 7, that the middle factor is not unique, in particular, the $k_j$'s are not unique.

Despite this non-uniqueness, the Wiener-Hopf type factorization is useful in studying the Fredholm properties of the operator $T_\Om$. Indeed, we have that $T_\Om$ is Fredholm if and only if $T_\Xi$ is Fredholm, and their indices coincide, as well as the dimensions of the kernels and codimensions of the range. In the next theorem we summarize the results of Theorem 1.3 and 1.4 in \cite{GtHJR21}.

\begin{theorem}
Let $\Om\in\Rat^{m\times m}$ with $\det \Om \not\equiv 0$ and Wiener-Hopf type factorization \eqref{WHfact1} as in Theorem \ref{T:fact} with $\Om_0=\diag_{j=1}^m (s_j/q_j)$ with $s_j,q_j\in\cP$ co-prime with roots only on $\BT$ and $z^{k_j}$ the $j$-th diagonal entry of $P_0$, for $j=1,\ldots, m$.
Then the following hold
\begin{itemize}
\item[{(i)}]
$T_\Om = T_{\Om_-} T_{z^{-k}}T_{\Om_0}T_{P_0} T_{\Om_+}$,
\item[{(ii)}]
$T_\Om$ is Fredholm if and only if $T_{\Om_0}$ is Fredholm, which happens exactly when  $s_1,\ldots,s_m$ are all constant,
\item[{(iii)}]
in case $T_\Om$ is Fredholm we have
\begin{align*}
\Index (T_\Om) & = mk + \Index(T_{\Om_0}) + \Index(T_{P_0})\\
& = mk + \sum_{j=1}^m \deg(q_j)  - \sum_{j=1}^m k_j.
\end{align*}
\end{itemize}
\end{theorem}

Unfortunately, due to the non-uniqueness and lower triangular form, rather than diagonal, of the middle term $\Xi$, it is not straightforward how the dimension of the kernel and the codimension of the range of $T_\Om$ can be determined from the Wiener-Hopf type factorization, which is possible in the case when there are no poles on the unit circle, due to the simpler nature of the Wiener-Hopf factorization in that case.

In this paper we shall determine the dimension of the kernel of $T_\Om$ in certain special cases where the kernel of $T_\Xi$ can be described and its dimension can be computed. Note that in the Wiener-Hopf type factorization, the middle factor $\Xi(z):=z^{-k} \Om_0(z) P_0(z)$ has the form
\begin{equation}\label{Xi}
\begin{aligned}
&\Xi(z)= z^{-k}\begin{pmatrix}
\frac{p_{11}(z)}{q_1(z)} & 0 & 0 & \cdots & 0\\
\frac{p_{21}(z)}{q_2(z)} & \frac{p_{22}(z)}{q_2(z)} & 0 & \cdots & 0 \\
\vdots & \vdots & \ddots & \ddots & \vdots \\
\vdots & \vdots & \ddots & \ddots & 0\\
\frac{p_{m1}(z)}{q_m(z)} & \frac{p_{m2}(z)}{q_m(z)}& \cdots & \frac{p_{m(m-1)}(z)}{q_m(z)}  & \frac{p_{mm}(z)}{q_m(z)}
\end{pmatrix}\\
&\mbox{$p_{ij},q_j\in\cP$, $q_i$ has zeroes only on $\BT$, $p_{ii}(z)=z^{k_i}$ for $1\leq j\leq i\leq m$}.
\end{aligned}
\end{equation}
The construction in \cite{GtHJR21} gives that, in addition, it is possible to arrange
\begin{equation}\label{XiExtra}
\deg(p_{ij})<k_i,\ 1\leq j\leq i\leq m
\quad\mbox{and}\quad q_i \mbox{ divides } q_{i+1},\ 1\leq i<m,
\end{equation}
which we will do in the remainder of the paper. However, even with these extra conditions, the middle factor is not unique.

\begin{definition}\label{D:SpecialForm}
A Wiener-Hopf type factorization \eqref{WHfact1} of $\Omega\in\Rat^{m \times m}$ is said to be in \emph{special form} if in $\Xi$ as in \eqref{Xi}, the largest difference between the degree of the numerator and the degree of the denominator in each column occurs on the diagonal of $\Xi(z)$, i.e., when
\begin{equation}\label{SpecialFormId}
k_n - \deg(q_n) \geq \deg(p_{jn}) - \deg(q_j), \qquad j = n+1,\ldots,m.
\end{equation}
\end{definition}

Note that if $\Om$ admits a Wiener-Hopf type factorization with middle term $\Xi$ as in \eqref{Xi} diagonal, with \eqref{XiExtra} in place, then this Wiener-Hopf type factorization is in special form.  It is unknown to us whether given any $\Om\in\Rat^{m \times m}$ with $\det \Om\not\equiv 0$ admits a Wiener-Hopf type factorization as in Theorem \ref{T:fact} with middle factor $\Xi$ in special form. On the other hand, we also do not have a counter-example to this. There are certain cases though for which  it is possible to obtain a Wiener-Hopf type factorization in special form from a Wiener-Hopf type factorization which is not; see Example 7.3 in \cite{GtHJR21} and Lemma \ref{L:2x2mono} below.

In the next theorem we state the main result of the paper.

\begin{theorem}\label{T:Main}
Let $\Om(z)\in\Rat^{m \times m}$ be a rational matrix function with poles on the unit circle $\BT$ given by a Wiener-Hopf type factorization \eqref{WHfact1} that is in special form, that is, $\Xi(z)=z^{-k} \Om_0(z) P_0(z)$ as in \eqref{Xi} satisfies \eqref{SpecialFormId}. Then
\[\dim \kernel T_{\Om} = \sum_{j=1}^m \max(k + \deg(q_j) - k_j, 0).\]
Furthermore, if in addition $\Omega(z)$ has no zeroes on the unit circle, then $T_\Om$ is Fredholm, and
\[\codim \Ran T_{\Om} = \sum_{j=1}^m \max(k_j - \deg(q_j) - k, 0).\]
\end{theorem}

In Theorem 1.4, we obtain the dimension of the kernel without assuming Fredholmness, but for the codimension of the range we assume Fredholmness, and then use the formula for the index. However, in the scalar case, when the operator is not Fredholm, we still get that the 
complement of the closure of the range is finite dimensional. It is unknown to us whether the formula for the codimension of the range of $T_\Om$ can also be proved without assuming Fredholmness.

The numbers
\[
\kappa_j:=k_j-\deg(q_j)-k,\quad j=1,\ldots,m,
\]
that appear in the formulas for the dimensions of the kernel and cokernel of the Fredholm operator $T_\Om$ in Theorem \ref{T:Main} are called the factorization indices of $T_\Om$. Using arguments similar to the proof for the classical case, we show that factorization indices are uniquely determined by $T_\Om$, in Proposition \ref{P:indices} below.

To illustrate the complications that occur in determining the dimension of the kernel of $T_\Om$, we analyse the case of $2 \times 2$ matrix functions in detail. In that case, without further assumptions on the matrix function, from a Wiener-Hopf type factorization we give descriptions of the kernel of $T_\Xi$, where $\Xi$ is the middle factor of the Wiener-Hopf type factorization, and use this description to provide a formula for the dimension of the kernel of $T_\Om$ in many cases,  and give an upper bound in the one case where we do not obtain a formula for the dimension of $\kernel T_\Om$. It is also shown that most of these do not correspond to a Wiener-Hopf type factorization that is in special form. The description of the kernel of $T_\Xi$ can partially be extended to the general case. However, given the many cases that already appear in the case of $2 \times 2$ matrix functions, we did not attempt to provide an analogous result for the general case, but rather restricted to a recursive procedure to determine $\dim \kernel T_\Om$.

We conclude the introduction with a brief summary of the paper. The paper consist of five sections, including the present introduction. In Section \ref{S:MiddleFactor} we provide some analysis of the middle factor $\Xi$ of a Wiener-Hopf type factorization, which we will need in subsequent sections. Theorem \ref{T:Main} and the fact that the factorization indices are uniquely determined are proved in Section \ref{S:MainTheorem}. The case of $2 \times 2$ matrix functions is studied in Section \ref{S:2x2}. There we give a description of the kernel of the Toeplitz type operator of the middle factor $\Xi$ of a Wiener-Hopf type factorization and determine the dimension of the kernel in many special cases, and give an upper bound in the one case where we do not obtain a concrete formula for the dimension. Finally, in Section \ref{S:KernChar} the description of the kernel of $T_\Xi$ from the $2 \times 2$ case is partially extended to the general case and a recursive procedure is presented to determine the dimension of the kernel.

\section{Some analysis of the middle factor $\Xi$}\label{S:MiddleFactor}

In this section we provide some analysis of the Toeplitz-like operator of the middle factor $\Xi$ in \eqref{MiddleFact} of the Wiener-Hopf type factorization \eqref{WHfact1}, specified further as in \eqref{Xi} and with the conditions \eqref{XiExtra} in place. We start with some general results, after which we prove Theorem \ref{T:Main}. Before turning to the kernel, we first provide a formula for the domain of $T_\Xi$, which can be viewed as the matrix analogue of the first part of Lemma 2.3 of \cite{GtHJR1}.

\begin{lemma}\label{L:MidtermDom}
Let $\Om\in\Rat^{m\times m}$ with poles on $\BT$ and Wiener-Hopf type factorization \eqref{WHfact1} with middle factor $\Xi$ in \eqref{MiddleFact} of the form \eqref{Xi} with conditions \eqref{XiExtra} in place. Then
\begin{equation}\label{DomXi}
\Dom(T_\Xi)=
\left\{f\in H^p_m \colon\!\!
\begin{array}{l}
\Xi f  =z^{-k} h + \eta \mbox{ with $h\in H^p_m,\, \eta\in\Rat_0^m(\BT)$}\!\!\!\! \\ \mbox{where $\eta_j=\frac{r_j}{q_j}$, $\deg(r_j)<\deg(q_j)$}
\end{array}
\right\}.
\end{equation}
\end{lemma}

\begin{proof}[\bf Proof]
The proof is similar to the proof in the scalar case \cite[Lemma 2.3]{GtHJR1}. We will provide a sketch of the proof for completeness. It is clear from the definition of the domain of $T_\Xi$ as in \eqref{TOm}, that the right hand-side of \eqref{DomXi} is contained in $\Dom(T_\Xi)$. Hence it remains to prove the reverse inclusion. Assume $f$ is in $\Dom(T_\Xi)$, that is, $\Xi f=g+\eta$ for some $g\in L^p_m$ and $\eta\in\Rat_0^m(\BT)$. Consider the $j$-th row of $\Xi f=g+\eta$:
\begin{equation}\label{DomCondJ}
\begin{aligned}
&z^{-k}\frac{p_{j1}(z)f_1(z)+\cdots + p_{jj}(z)f_{j}(z)}{q_j(z)}=g_j(z)+\eta_j(z).
\end{aligned}
\end{equation}
Write $q_j \eta_{j}=\rho_j+r_j$ for $\rho_j\in\Rat_0(\BT)$ and $r_j\in\cP_{\deg(q_j)-1}$ and multiply both sides of \eqref{DomCondJ} with $q_j$. This yields
\[
z^{-k}(p_{j1}(z)f_1(z)+\cdots + p_{jj}(z)f_{j}(z))=q_j(z)g_j(z)+ \rho_j(z)+r_j(z).
\]
Note that all elements apart from $\rho_j$ are in $L^p$, which means that $\rho_j$ must be in $L^p$ too. However, $\Rat_0(\BT)\cap L^p=\{0\}$, by \cite[Lemma 2.2]{GtHJR1}, hence $\rho_j=0$. Replacing $\rho_j$ by $0$ and multiplying with $z^k$ on both sides we get
\[
p_{j1}(z)f_1(z)+\cdots + p_{jj}(z)f_{j}(z)=z^k q_j(z)g_j(z)+ z^k r_j(z).
\]
Now observe that all elements apart from $z^k q_j(z)g_j(z)$ are in $H^p$, which means that $z^k q_j(z)g_j(z)$ must be in $H^p$ as well. Since $q_j(z)\neq 0$ for all $z\in\BD$ and $z^kg_j(z)$ is in $L^p$, this further implies that $z^k g_j(z)\in H^p$, or, equivalently, $g_j\in z^{-k} H^p$. Hence we have shown that $g_j\in z^{-k} H^p$ and $\eta_j=r_j/q_j$ with $\deg(r_j)<\deg(q_j)$. Since this holds for $j=1,\ldots,m$, it follows that $f$ is contained in the left hand side of \eqref{DomXi}, and the proof is complete.
\end{proof}

In the scalar case, for $\om(z)=z^k/q(z)\in\Rat(\BT)$ we have $\kernel T_{\om}=\cP_{\deg(q)-k-1}$, by \cite[Lemma 4.1]{GtHJR1}. In the next result we show that for the middle factor $\Xi$ it is still the case that the kernel consists of polynomials only, and we provide a bound on the degree.

\begin{lemma}\label{L:polykern}
Let $\Om\in\Rat^{m\times m}$ with poles on $\BT$ and Wiener-Hopf type factorization \eqref{WHfact1} with middle factor $\Xi$ in \eqref{MiddleFact} of the form \eqref{Xi} with conditions \eqref{XiExtra} in place. Then $\kernel T_\Xi\subset \cP^m$ and for $f=(f_1,\ldots,f_m)^T\in\kernel T_\Xi$ we have
\begin{align}
\deg(f_j)&< \max(\deg(f_1),\ldots,\deg(f_{j-1}),\deg(q_j)+k-k_j)\label{degineq}\\
&\leq \max(\deg(q_1)-k_1-1,\dots,\deg(q_{j-1})-k_{j-1}-1,\deg(q_j)-k_j)+k.\notag
\end{align}
\end{lemma}

\begin{proof}[\bf Proof]
Let $f=(f_1,\cdots,f_m)^{T}\in \kernel T_{\Xi}$ and let $j\in \{1,\ldots,m\}$. Assume that for $i=1,\ldots,j-1$ we have $f_i\in \cP$ and the inequalities in \eqref{degineq} hold for $f_i$. We show that also $f_j\in \cP$ and inequality in \eqref{degineq} hold for $f_j$. Since $f\in \kernel T_{\Xi}$, in particular, $f\in\Dom(T_\Xi)$, and by Lemma \ref{L:MidtermDom} we have $\Xi(z) f(z)= z^{-k}h(z)+\rho(z)$ with $h=(h_1,\ldots,h_m)^T\in H^P_m$, and $\rho=(r_1/q_1,\ldots,r_m/q_m)^{T}\in \Rat_0^m(\BT)$, i.e., with $r_j\in\cP_{\deg(q_j)-1}$. Write $h$ as $h(z)=z^kg(z)+v(z)$ with $v\in\cP^m_{k-1}$ and $g\in H^p_m$. Then,
from $T_\Xi f=0$ we see that $g=0$,
which gives that $\Xi(z) f(z)= z^{-k}v(z)+\rho(z)$ with $v=(v_1,\ldots,v_m)^T\in \cP_{k-1}^m$
Now consider the $j$-th entries on both sides of the identity $\Xi(z) f(z)= z^{-k}v(z)+\rho(z)$ and multiply with $z^k q_j(z)$. This yields
\[
p_{j1}(z)f_1(z) + \cdots + p_{jj}(z)f_{j}(z)= v_j(z)+z^k r_j(z),\quad \mbox{with}\quad p_{jj}(z)=z^{k_j}.
\]
Since $f_1,\ldots,f_{j-1},p_{j1},\ldots p_{j(j-1)}, r_j$ and $v_j$ are polynomials, it is clear that $p_{jj}f_j$ must be a polynomial too. Moreover, since $p_{jj}(z)=z^{k_j}$, we have
\[
f_j(z)=\frac{z^{k}r_j(z)-(p_{j1}(z)f_1(z) + \cdots + p_{j (j-1)}(z)f_{j-1}(z))}{z^{k_j}}\in\Rat \cap H^p.
\]
The fact that $f_j\in H^p$, means that $f_j$ cannot have a pole at $0$, so that $z^{k_j}$ must be a factor of $z^{k}r_j(z)-(p_{j1}(z)f_1(z) + \cdots + p_{j (j-1)}(z)f_{j-1}(z))$. That implies that $f_j$ is in $\cP$ with
\[
\deg(f_j)\leq \max(\deg(p_{j1}f_1),\ldots,\deg(p_{j (j-1)}f_{j-1}),\deg(r_j)+k)-k_j.
\]
Note that $\deg(r_j)<\deg(q_j)$ and
\begin{align*}
\deg(p_{ji}f_i)=\deg(p_{ji}) + \deg(f_i)< k_j +\deg(f_i).
\end{align*}
Hence
\[
\deg(f_j)<\max(\deg(f_1),\ldots,\deg(f_{j-1}),\deg(q_j)+k-k_j)
\]
as claimed. Thus the first inequality of \eqref{degineq} holds.

It remains to prove the second inequality of \eqref{degineq}. For $j=1$ it follows immediately that this holds, in fact, with equality. For $j>1$ it remains to show that
\begin{align}
&\max(\deg(f_1),\ldots,\deg(f_{j-1}))\leq\notag \\
&\qquad\qquad \leq\max(\deg(q_1)-k_1-1,\ldots,\deg(q_{j-1})-k_{j-1}-1)+k,\label{ineqinduc}
\end{align}
where by assumption the inequality holds if $j$ is replaced by $i=1,\ldots,j-1$. Applying the first inequality in \eqref{degineq} to $f_{j-1}$, it follows that
\begin{align*}
&\max(\deg(f_1),\ldots,\deg(f_{j-2}),\deg(f_{j-1}))\leq\\
&\qquad \leq\max(\deg(f_1),\ldots,\deg(f_{j-2}),\\
&\qquad \qquad \max(\deg(f_1)-1,\cdots,\deg(f_{j-2})-1,\deg(q_{j-1})+k-k_{j-1}-1))\\
&\qquad\qquad = \max(\deg(f_1),\ldots,\deg(f_{j-2}), \deg(q_{j-1})+k-k_{j-1}-1).
\end{align*}

Now simply use the upper bound on $\max(\deg(f_1),\ldots,\deg(f_{j-2}))$ from the induction hypothesis to conclude that \eqref{ineqinduc} holds, and hence the second inequality in \eqref{degineq}.
\end{proof}

Since the kernel of $T_\Xi$ is contained in $\cP^m$, it follows that the scalar entries in $f\in\kernel T_\Xi$ are in the domain of the Toeplitz operators of the scalar entries of $\Xi$, because by \cite[Proposition 2.1]{GtHJR1} $\cP$ is always contained in their domain. As a consequence, we can formulate the kernel condition of $T_\Xi$ in terms of the Toeplitz-like operators corresponding to the scalar entries of $\Xi$.

\begin{corollary}\label{C:Kern1}
Let $\Om\in\Rat^{m\times m}$ with poles on $\BT$ and Wiener-Hopf type factorization \eqref{WHfact1} with middle factor $\Xi$ in \eqref{MiddleFact} in the form \eqref{Xi} with conditions \eqref{XiExtra} in place. For $f=(f_1,\cdots,f_m)^T\in\kernel T_\Xi$ we have
\[
0= T_\Xi f
=\pmat{T_{z^{-k}\frac{p_{11}}{q_1}}f_1\\
T_{z^{-k}\frac{p_{21}}{q_2}}f_1+ T_{z^{-k}\frac{p_{22}}{q_2}}f_2 \\
\vdots \\
T_{z^{-k}\frac{p_{m1}}{q_m}}f_1 +  \cdots + T_{z^{-k}\frac{p_{m(m-1)}}{q_m}} f_{m-1}  + T_{z^{-k}\frac{p_{mm}}{q_m}}f_m}.
\]
\end{corollary}

\begin{proof}[\bf Proof]
Since $f=(f_1,\cdots,f_m)^{T}\in\kernel T_\Xi$, it follows from Lemma \ref{L:polykern} that each $f_j$ is in $\cP$ and hence the right-hand side of the formula is well defined, because $\Dom(T_{\frac{p_{ij}}{q_i}})$ contains $\cP$ for all $1\leq i\leq j\leq m$ by \cite[Proposition 2.1]{GtHJR1}. That the identity holds is because $H^p$, $K^p:=\ov{\textup{span}}\{e^{ikt}\colon k\in\{-1,-2,\ldots\} \}\subset L^p$ and $\Rat_0(\BT)$ are vector spaces. Indeed, for the $k$-th entry one needs to inspect
\[
z^{-k}\left(\frac{p_{k1}(z)}{q_k(z)}f_1(z) + \cdots + \frac{p_{kk}(z)}{q_k(z)}f_{k}(z)\right).
\]
However, since for each $j$, $f_j$ is in the domain of $T_{z^{-k}\frac{p_{kj}}{q_k}}$, the $j$-th summand $z^{-k}\frac{p_{kj}(z)}{q_k(z)}f_j(z)$ is equal to $g_j+h_j+\rho_j$ with $g_j\in K^{p}$, $h_j\in H^p$, $\rho_j\in \Rat_0(\BT)$ all uniquely determined and with $h_j=T_{z^{-k}\frac{p_{kj}}{q_k}}f_j$. The sum itself then also comes in the form of a sum of an element $g_1+\cdots+g_k\in K^{p}$, $h_1+\cdots+h_k\in H^p$ and $\rho_1+\cdots+\rho_k\in \Rat_0(\BT)$ and the function $h_1+\cdots+h_k$ corresponds to the $k$-th entry in $T_\Xi f$ while also corresponding to the $k$-th entry in the right-hand side of the equation. Hence the identity follows.
\end{proof}

\section{Proof of Theorem \ref{T:Main} and determination of the factorization indices}\label{S:MainTheorem}

We begin this section with a proof of the first part of Theorem \ref{T:Main}.

\begin{proposition}\label{P:kernel2}
Let $\Om\in\Rat^{m\times m}$ with poles on $\BT$ and Wiener-Hopf type factorization \eqref{WHfact1} with middle factor $\Xi$ in \eqref{MiddleFact} in the form \eqref{Xi} with conditions \eqref{XiExtra} in place.  Assume that the factorization of $\Om$ is in special form (as in Definition \ref{D:SpecialForm}). Then
\[
\dim \kernel T_{\Om} =\dim \kernel T_\Xi  = \sum_{j=1}^m \max(k + \deg(q_j) - k_j, 0).
\]
In fact,
\[
\kernel T_\Xi=\oplus_{j=1}^m \cP_{k+\deg(q_j) -k_j-1}.
\]
\end{proposition}

\begin{proof}[\bf Proof]
Let $\Xi$ be the middle factor in the Wiener-Hopf type factorization of $\Omega$, given by \eqref{MiddleFact}-\eqref{Xi}. By assumption, the conditions of Definition \ref{D:SpecialForm} hold, in particular, we have \eqref{SpecialFormId}. The result follows once we have proved the formula for $\kernel T_\Xi$, which we will prove by induction. More precisely, we will prove by induction that for $f=(f_1\ldots,f_m)^T\in\Dom(T_\Xi)$ the first $l$ entries of $T_\Xi f$ are zero if and only if $f_j\in  \cP_{k+\deg(q_j) -k_j-1}$ for $j=1,\ldots,l$. To this end, let $f=(f_1, \ldots , f_m)^{T}\in\Dom(T_\Xi)$, so that we have $\Xi(z) f(z)=h(z)+z^{-k}v(z)+\rho(z)$ with $h=(h_1, \ldots , h_m)^{T}\in H_m^p$, $v=(v_1,\ldots,v_m)^{T}\in \cP^m_{k-1}$ and $\rho=(r_1/q_1,\ldots,r_m/q_m)^{T} \in \Rat_0^m(\BT)$, i.e., with $r_j\in\cP_{\deg(q_j)-1}$, according to Lemma \ref{L:MidtermDom}. Now assume $l\leq m$ such that the claim holds for $j=1,\ldots,l-1$.

First assume that the first $l$ entries of $T_\Xi f$ are zero, i.e., $h_1=\cdots=h_l=0$. In particular, since the first $l-1$ entries are zero, we know that
$\deg(f_j)<\deg(q_j)+k -k_j$ for $j=1,\ldots,l-1$, by the induction assumption. Multiplying the $l$-th component of $\Xi(z) f(z)=h(z)+z^{-k}v(z)+\rho(z)$ with $z^k q_l(z)$ yields
\begin{equation}\label{kerneq l}
p_{l1}(z)f_1(z) + \cdots  + p_{ll}(z)f_l(z) = q_l(z) v_l(z) + z^k r_l(z).
\end{equation}
Since $\deg(r_l)<\deg(q_l)$  and $\deg(v_l)<k$, we have
\[
\deg(q_l(z) v_l(z) + z^k r_l(z))<\deg(q_l)+k.
\]
The induction assumption gives $\deg(f_j)<\deg(q_j)+k -k_j$. Together with the condition \eqref{SpecialFormId}, for $j=1,\ldots,l-1$ this implies that
\[
\deg(p_{l j}f_j)=\deg(p_{l j})+\deg(f_j) < \deg(p_{l j}) + \deg(q_j)+k-k_j  \leq \deg(q_l) +k.
\]
Thus, the right hand-side and the first $l-1$ summands in the left-hand side of \eqref{kerneq l} are all polynomials of degree less than $\deg(q_l)+k$. Then also $p_{ll(z)}f_l(z)=z^{k_l}f_l(z)$ must be a polynomial of degree less than $\deg(q_l)+k$, which implies that $f_l\in \cP$ with $\deg(f_l)<\deg(q_l)+k-k_l$.

Conversely, assume that $\deg(f_j)<\deg(q_j)+k -k_j$ for $j=1,\ldots,l$. The recursion assumption implies that the first $l-1$ entries of $T_\Xi f$ are zero. Again, multiplying the $l$-th component of $\Xi(z) f(z)=h(z)+z^{-k}v(z)+\rho(z)$ with $z^k q_l(z)$ yields
\begin{equation}\label{kerneq l2}
p_{l1}(z)f_1(z) + \cdots + p_{ll}(z)f_l(z) = z^kq_l(z)h_l(z) + q_l(z) v(z) +  z^k r_l(z).
\end{equation}
Since $p_{ll}(z)=z^{k_l}$ and $\deg(f_l)<\deg(q_l)+k -k_j$, we have $\deg(p_{ll}f_l)<\deg(q_l)+k$. As above, $\deg(q_l(z) v_l(z) + z^k r_l(z))<\deg(q_l)+k$ and from \eqref{SpecialFormId} we again obtain that $\deg(p_{l j}f_j)< \deg(q_l)+k$. The identity then implies that $z^kq_l(z)h_l(z)\in \cP$ with $\deg(z^kq_lh_l)<\deg(q_l)+k$. However, that can only happen when $h_l=0$, which implies that the $l$-th entry in $T_\Xi f$ is zero as well. Since the statement holds for $l=1,\ldots,m$, the proof is complete.
\end{proof}

In addition to the first claim of Theorem \ref{T:Main}, in the above proposition we have also shown that the kernel of $T_\Xi$ splits as the direct sum of $m$ subspaces of $H^p$ in the case that the factorization of $\Om$ is in special form. This, however, does not imply that the domain of $T_\Xi$ splits in a similar fashion.

Next we prove the second claim of Theorem \ref{T:Main}.

\begin{proposition}\label{P:MainPt2}
Let $\Om\in\Rat^{m\times m}$ with poles on $\BT$ and Wiener-Hopf type factorization \eqref{WHfact1} with middle factor $\Xi$ in \eqref{MiddleFact} in the form \eqref{Xi} with conditions \eqref{XiExtra} in place.  Assume that the factorization of $\Om$ is in special form (as in Definition \ref{D:SpecialForm})and that $T_\Om$ is Fredholm. Then
\[
\dim \kernel T_{\Om} = \sum_{j=1}^m \max(k + \deg(q_j) - k_j, 0)
\]
and
\[
\codim \Ran T_{\Om} = \sum_{j=1}^m \max(k_j - \deg(q_j) - k, 0).\]
\end{proposition}

\begin{proof}[\bf Proof]
The formula for $\dim \kernel T_{\Om}$ was already established in Proposition \ref{P:kernel2}. Since $T_\Om$ is assumed to be Fredholm, according to Theorem 1.4 in \cite{GtHJR21}, the Fredholm index is given by
\[
\Index T_{\Om} = \dim \kernel T_{\Om} - \codim \Ran T_{\Om} =km + \sum_{j=1}^m \deg(q_j)  - \sum_{j=1}^m k_j.
\]
Using the formula for $\dim \kernel T_{\Om}$ from Proposition \ref{P:kernel2} it then follows that
\begin{align*}
&\codim \Ran T_{\Om}
 = \dim \kernel T_{\Om} - \Index T_{\Om}\\
&\qquad\quad =  \sum_{j=1}^m \max(k + \deg(q_j) - k_j, 0) - ( km + \sum_{j=1}^m \deg(q_j)  - \sum_{j=1}^m k_j )\\
&\qquad\quad = \sum_{j=1}^m \max(k_j - \deg(q_j) - k , 0).\qedhere
\end{align*}
\end{proof}

The numbers
\begin{equation}\label{kappaj}
\kappa_j:=k_j-\deg(q_j)-k,\quad j=1,\ldots,m,
\end{equation}
that appear in the formulas for the dimensions of the kernel and cokernel of the Fredholm operator $T_\Om$ are called the {\em factorization indices} of $T_\Om$. Note that we can now write
\[
\dim \kernel T_{\Om} = \sum_{\kappa_j\leq 0} - \kappa_j \ands
\codim \Ran T_{\Om} = \sum_{\kappa_j\geq 0} \kappa_j.
\]
Although the Wiener-Hopf type factorization \eqref{WHfact1} of $\Om$ is far from unique, following an argument on page 590 in \cite{GGK2}, we can show that factorization indices are uniquely determined by $T_\Om$, provided that we can find a Wiener-Hopf type factorization in special form.

\begin{proposition}\label{P:indices}
Let $\Om\in\Rat^{m\times m}$ with poles on $\BT$ and Wiener-Hopf type factorization \eqref{WHfact1} with middle factor $\Xi$ in \eqref{MiddleFact} in the form \eqref{Xi}.  Assume that the factorization of $\Om$ is in special form (as in Definition \ref{D:SpecialForm}) and that $T_\Om$ is Fredholm. Let $\kappa_1,\ldots,\kappa_m$ be the factorization indices given by \eqref{kappaj}. Then for each integer $\kappa$ we have
\begin{align*}
\#\{j\colon \kappa_j
=\kappa\} &= \codim \Ran(T_{z^{1-\kappa}\Om}) - 2\,\codim \Ran(T_{z^{-\kappa}\Om}) +\\
&\qquad\qquad + \codim \Ran(T_{z^{-\kappa-1}\Om}).
\end{align*}
\end{proposition}

\begin{proof}[\bf Proof] For $\kappa\in\BZ$, set $\Om_\kappa(z):=z^{-\kappa}\Om(z)\in \Rat^{m\times m}$. Since $T_\Om$ is Fredholm, by assumption, so is $T_{\Om_\kappa}$, since no zeroes on $\BT$ are introduced. Moreover, $\Om_\kappa$ has a Wiener-Hopf type factorization \eqref{WHfact1} with the same plus function $\Om_+(z)$ and minus function $\Om_-(z)$ and with middle term $z^{-\kappa}\Xi(z)$. Note that the Wiener-Hopf type factorization of $\Om_\kappa$ is also in special form and has factorization indices
\[
\wtil{\kappa}_j= k_j-\deg(q_j)-k-\kappa,\quad j=1,\ldots,m.
\]
Hence, we have
\[
d(T_{\Om_\kappa}):=\codim \Ran T_{\Om_\kappa} = \sum_{\kappa_j\geq \kappa} \kappa_j - \kappa = \sum_{\kappa_j> \kappa} \kappa_j - \kappa =\sum_{\kappa_j\geq \kappa+1} \kappa_j - \kappa.
\]
This implies that
\begin{align*}
d(T_{\Om_\kappa})-d(T_{\Om_{\kappa+1}})
&= \sum_{\kappa_j\geq \kappa} \kappa_j - \kappa - \sum_{\kappa_j\geq \kappa +1} \kappa_j - \kappa -1\\
&=\sum_{\kappa_j\geq \kappa+1} \kappa_j - \kappa - \sum_{\kappa_j\geq \kappa +1} \kappa_j - \kappa -1\\
&=\sum_{\kappa_j\geq \kappa +1} 1 = \#\{j \colon \kappa_j\geq \kappa +1\}.
\end{align*}
Therefore, we have
\begin{align*}
\#\{j\colon \kappa_j=\kappa\}
&= \#\{j \colon \kappa_j\geq \kappa \} - \#\{j \colon \kappa_j\geq \kappa +1\}\\
&=(d(T_{\Om_{\kappa-1}})-d(T_{\Om_{\kappa}})) - (d(T_{\Om_\kappa})-d(T_{\Om_{\kappa+1}}))\\
&=d(T_{\Om_{\kappa-1}}) - 2 d(T_{\Om_{\kappa}})) + d(T_{\Om_{\kappa+1}}),
\end{align*}
which proves our claim.
\end{proof}

\section{The $2\times 2$ case}\label{S:2x2}

In this section we consider the general $2 \times 2$ case. Hence, let $\Om\in\Rat^{2 \times 2}$ with poles on $\BT$ with a Wiener-Hopf type factorization \eqref{WHfact1} with middle factor $\Xi$ in \eqref{MiddleFact} of the form \eqref{Xi}, that is
\begin{equation}\label{2x2}
\Xi(z)=\pmat{\xi_{11}(z)&0\\\xi_{21}(z)&\xi_{22}(z)}
=z^{-k}\begin{pmatrix}
\frac{z^{k_1}}{q_1(z)} & 0 \\[2mm]
\frac{z^{k_{21}}d_{21}(z)}{q_2(z)} & \frac{z^{k_2}}{q_2(z)}
\end{pmatrix},
\end{equation}
where we also make the assumption in \eqref{XiExtra}:
\begin{equation}\label{2x2con}
k_{21} +\deg (d_{21}) < k_2,\quad d_{21}(0)\not= 0,\quad q_1 | q_2.
\end{equation}
We do not assume that the Wiener-Hopf type factorization of $\Om$ is in special form, but we will discuss this case in Corollary \ref{C:2x2special} below.

As explained in the general case, $\dim \kernel T_\Om=\dim \kernel T_\Xi$, hence it suffices to analyse the kernel of $T_\Xi$. Using Corollary \ref{C:Kern1} we can make a few direct observations regarding $\kernel T_\Xi$:
\begin{itemize}
\item[(i)] If $\kernel T_{\xi_{11}}=\{0\}$, or, equivalently, $k_1\geq \deg(q_1)+k$, then
\[
\kernel T_\Xi= \{0\}\oplus \kernel T_{\xi_{22}}=\kernel T_{\xi_{11}} \oplus \kernel T_{\xi_{22}}.
\]

\item[(ii)] It is always the case that $\{0\}\oplus \kernel T_{\xi_{22}} \subset \kernel T_\Xi$. However, if $\kernel T_{\xi_{22}}\neq\{0\}$, or, equivalently, $k_2<\deg(q_2)+k$, then $T_{\xi_{22}}$ is surjective and with an appropriate choice in a complement of $\kernel T_{\xi_{22}}$ one can always make the second row zero, so that
\[
\kernel T_\Xi= \kernel T_{\xi_{11}} \oplus \kernel T_{\xi_{22}}.
\]
If $k_2=\deg(q_2)+k$, then $\kernel T_{\xi_{22}}=\{0\}$ but $T_{\xi_{22}}$ is still surjective and the same argument still leads to the same formula for $\kernel T_\Xi$.
\end{itemize}
In both cases, $k_1\geq \deg(q_1)+k$ and $k_2 \leq \deg(q_{2})+k$, we obtain
\[
\dim \kernel T_\Xi=\max(\deg(q_1)+k-k_1,0) + \max(\deg(q_2)+k-k_2,0),
\]
in correspondence with the result for the case where $\Xi$ has a factorization in special form.

Next we present a general result that expresses the kernel of $T_\Xi$ in terms of the kernel of a certain compressed Toeplitz matrix. To express the result we make an identification between $\cP_{l-1}$ and $\BC^l$ via the following notation:
\[
\vec{g}=(g_0,\ldots,g_{l-1})^T\in\BC^{l}\mbox{ for }g(z)=g_0+g_1 z+ \cdots + g_{l-1}z^{l-1}\in\cP_{l-1}.
\]

\begin{proposition}\label{P:2x2kernel1}
Let $\Xi(z)$ be given by \eqref{2x2} such that \eqref{2x2con} is satisfied. For all $f=(f_1,f_2)^T\in H^p_{2}$ the following are equivalent:
\begin{itemize}
\item[(i)] $f\in \kernel T_\Xi$;

\item[(ii)] $f_1,f_2\in\cP$ such that
\begin{align*}
&\deg(f_1)<\deg(q_1)+k-k_1 \mbox{ and}\\ &\qquad\deg(d_{21}(z)f_1(z)+z^{k_2-k_{21}}f_2(z))<\deg(q_2)+k-k_{21};
\end{align*}

\item[(iii)] one of the following two cases occurs:
\begin{itemize}
  \item[(iii.a)] if $k_1\geq \deg(q_1)+k$ or $k_2\leq \deg(q_2)+k$, then $f_1,f_2\in\cP$ such that $\deg(f_1)<\deg(g_1)+k-k_1$ and $\deg(f_2)<\deg(g_2)+k-k_2$;

  \item[(iii.b)] if $k_1< \deg(q_1)+k$ and $k_2> \deg(q_2)+k$, let $M$ and $N$ be the compressions of the Toeplitz operator of $d_{21}$ obtained by restricting to the first $\deg(q_1)+k-k_1$ columns, in both cases, and selecting rows
\[
\max(\deg(q_2)+k-k_{21},0)+1\quad \mbox{to}\quad k_2-k_{21}\quad \mbox{for $M$},
\]
and rows
\[
k_2-k_{21}+1 \quad \mbox{to}\quad \max(k_2-k_{21},\deg(q_1)+k-k_1+\deg(d_{21})) \quad \mbox{for $N$},
\]
with $N$ to be interpreted as vacuous if the upper bound is less than the lower bound, then $f_1\in\cP$ such that $\deg(f_1)<\deg(q_1)+k-k_1$ and $M \vec{f}_1=0$ and $f_2$ is uniquely determined by $\vec{f}_2=-N\vec{f}_1$.
\end{itemize}
\end{itemize}
\end{proposition}

\begin{remark}\label{R:2x2dichotomy}
Note that there is a certain dichotomy between the two subcases (iii.a) and (iii.b). In the first there is some freedom in $f_2$ (provided $k_2< \deg(q_2)$) but the second condition in (ii) does not put any further constraint on $f_1$, while in the second subcase $f_1$ is further constrained by the condition $M \vec{f}_1=0$, but $f_2$ is uniquely determined by the choice of $f_1$.
\end{remark}

\begin{proof}[\bf Proof of Proposition \ref{P:2x2kernel1}]
First we prove the equivalence of (i) and (ii). Note that, by Lemma \ref{L:polykern}, $\kernel T_{\Xi}$ is contained in $\cP^2$, so that in both (i) and (ii) we have that $f_1,f_2\in\cP$. In particular, $f\in \Dom(T_\Xi)$ in both cases, so we may assume this to be the case. Then, by Lemma \ref{L:MidtermDom}, it follows that
\begin{align*}
&\qquad\qquad\Xi(z) f(z)= h(z)+z^{-k}v(z)+\rho(z) \mbox{ with}\\
& h=\pmat{h_1\\h_2} \in H^p_2,\, v=\pmat{v_1\\v_2} \in \cP^2_{k-1},\, \rho=\pmat{\rho_1\\\rho_2}=\pmat{r_1/q_1\\ r_2/q_2}\in\Rat_0(\BT),
\end{align*}
and it follows that $f\in \kernel T_\Xi$ precisely when $h=0$. Writing out the identity in the first row it follows that we get $h_1=0$ if and only if $f_1\in \kernel T_{\xi_{11}}$, that is, $\deg(f_1)<\deg(q_1)+k-k_1$. Considering the second row, it follows that $h_2=0$ if and only if
\[
\frac{z^{k_{12}}d_{21}(z)}{z^k q_2(z)}f_1(z) + \frac{z^{k_2}}{z^k q_2(z)}f_2(z)=\frac{v_2(z)}{z^k} + \frac{r_2(z)}{q_2(z)}
\]
for some $r_2,v_2\in\cP$ with $\deg(r_2)<\deg(q_2)$ and $\deg(v_2)<k$. Multiplying by $z^k q_2(z)$ on both sides and factoring out $z^{k_{21}}$ on the left side, we see that $h_1=0$ is equivalent to
\begin{equation}\label{degcon}
z^{k_{21}}(d_{21}(z)f_1(z) + z^{k_2-k_{21}}f_2(z))=q_2(z)v_2(z)+ z^k r_2(z),
\end{equation}
for some $r_2,v_2\in\cP$ with $\deg(r_2)<\deg(q_2)$ and $\deg(v_2)<k$. Note that $q_2$ and $z^k$ are co-prime. Therefore, ranging over $r_2\in\cP_{\deg(q_2)-1}$ and $v_2\in\cP_{k-1}$, in the right-hand side of \eqref{degcon} we can get any vector in $\cP_{\deg(q_2)+k-1}$. As a result, we see that $h_2=0$ if and only if
\[
z^{k_{21}}(d_{21}(z)f_1(z) + z^{k_2-k_{21}}f_2(z))\in\cP_{\deg(q_2)+k-1},
\]
or, what is the same,
\[
\deg(d_{21}(z)f_1(z) + z^{k_2-k_{21}}f_2(z))<\deg(q_2)+k-k_{21}.
\]
This proves that (i) and (ii) are equivalent.

Next we prove the equivalent of (ii) and (iii). The cases $k_1\geq \deg(q_1)+k$ or $k_2\leq \deg(q_2)+k$ in (iii.a) correspond precisely to the two cases discussed after \eqref{2x2con}. Hence we may restrict to the special case of item (iii.b) and assume that $k_1<\deg(q_1)+k$ and $k_2>\deg(q_2)+k$, which we will do in the remainder of the proof.

In both (ii) and (iii.b), $f_1$ is constrained by $\deg(f_1)<\deg(q_1)+k-k_1$, and hence we may assume that to be the case. Let $T$ denote the restriction of the Toeplitz operator of $d_{21}$ to the first $\deg(q_1)+k-k_1$ columns and let $L$ be the matrix obtained by selecting the first $\deg(q_2)+k-k_{21}$ rows of $T$, corresponding to a vacuous matrix (no rows) in case $\deg(q_2)+k\leq k_{21}$. Hence $L$ and $N$ can both be vacuous, but $M$ cannot since $k_2$ is larger than both $k_{21}$ and $\deg(q_2)+k$. Note that after the $\deg(q_1)+k-k_1+\deg(d_{21})$-th row, $T$ only has zero-rows. Set $g=d_{21} f_1$. Then $\deg(g)< \deg(d_{21})+\deg(q_1)+k-k_1$. For $f_2\in\cP$, write $(\vec{f}_2)_1$ for the first
\begin{align*}
&\max(k_2-k_{21},\deg(q_1)+k-k_1+\deg(d_{21}))- (k_2-k_{21})=\\
&\qquad\qquad= \max(0,\deg(q_1)+k-k_1+\deg(d_{21})+k_{21}-k_2)
\end{align*}
entries of $\vec{f}_2$, which might be a vacuous vector, and
$(\vec{f}_2)_2$ for the remaining entries. Then
\[
\overrightarrow{d_{21}(z)f_1(z)+z^{k_2-k_{21}}f_2(z)}=\vec{g}+\overrightarrow{z^{k_2-k_{21}}f_2(z)}
=\pmat{L\vec{f}_1\\M\vec{f_1}\\N \vec{f}_1+(\vec{f}_2)_1\\(\vec{f}_2)_2}.
\]
Now it follows that $\deg(d_{21}(z)f_1(z)+z^{k_2-k_{21}}f_2(z))<\deg(q_2)-k_{21}$ corresponds to all but the first component being zero, that is, $(\vec{f}_2)_2=0$, $N \vec{f}_1+(\vec{f}_2)_1=0$ and $M\vec{f}_1=0$. In other words, the second condition of item (ii) holds if and only if $M\vec{f}_1=0$ and $\vec{f}_2=(\vec{f}_2)_1=-N \vec{f}_1$, which proves that (ii) and (iii) are equivalent.
\end{proof}

It follows from statement (iii) that the dimension of $T_\Xi$ in subcase (b) is equal to the dimension of the kernel of the matrix $M$. We will next show how this dimension can be computed, at least in certain cases.

\begin{proposition}\label{P:2x2kernel2}
Let $\Xi$ be given by \eqref{2x2} such that \eqref{2x2con} is satisfied. Assume that $k_2> \deg(q_2)+k$ and $k_1< \deg(q_1)+k$. Then regarding the dimension of $\kernel T_\Xi$ we have the following:
\begin{itemize}
\item[(i)] If $k_2-k_{21}\leq \deg(q_1)+k-k_1$, then
\[
\dim \kernel T_\Xi= \deg(q_1)+k-k_1 + \max(\deg(q_2)+k-k_2,k_{21}- k_2).
\]
In particular,
\begin{itemize}
\item[(i.a)] if in addition $\deg(q_2)+k-k_{21}\geq 0$, then
\[
\dim \kernel T_\Xi=\deg(q_1)-k_1 + \deg(q_2) - k_2 + 2k;
\]

\item[(i.b)] if in addition $\deg(q_2)+k-k_{21}< 0$, then
\[
\dim \kernel T_\Xi=\deg(q_1)+k-k_1 + k_{21} - k_2.
\]

\end{itemize}

\item[(ii)] If $k_2-k_{21}> \deg(q_1)+k-k_1$  and in addition $\deg(q_2)+k\leq k_{21}$, then $\dim \kernel T_\Xi=0$.

\item[(iii)] If $k_2-k_{21}> \deg(q_1)+k-k_1$ and $\deg(q_2+k)> k_{21}$ and in addition $\deg(d_{21})\leq  \deg(q_2)+k- k_{21}$ or
$\deg(d_{21})\leq k_1+  k_2- \deg(q_1)-k -k_{21}$, then $\dim \kernel T_\Xi = \max(a_1,a_2,0)$, where
\begin{align*}
 &a_1= \deg(q_1)-k_1+\deg(q_2)-k_2 +2k \mbox{ and}\\
 &\qquad   a_2=\min(\deg(q_2)+k-k_{21}-\deg(d_{21}), \deg(q_1)+k-k_1).
\end{align*}
In particular:
\begin{itemize}
\item[(iii.a)] if in addition
\[
\deg(q_2)+k-k_{21} \geq \deg(d_{21}) \geq k_1+k_2-\deg(q_1)-k-k_{21},
\]
then $\dim \kernel T_\Xi= \deg(q_1)-k_1 + \deg(q_2)-k_2+2k$;

\item[(iii.b)] if in addition
\[
\deg(q_2)+k-k_{21} \geq \deg(d_{21}) \leq k_1+k_2-\deg(q_1)-k-k_{21},
\]
then $\dim \kernel T_\Xi=\min(\deg(q_1)+k-k_1, \deg(q_2)+k-k_{21}- \deg(d_{21}));$

\item[(iii.c)] if in addition
\[
\deg(q_2)+k-k_{21} \leq \deg(d_{21}) \leq k_1+k_2-\deg(q_1)-k-k_{21},
\]
then $\dim \kernel T_\Xi=0$.

\end{itemize}

\item[(iv)] If $k_2-k_{21}> \deg(q_1)+k-k_1$ and $\deg(q_2)+k> k_{21}$ and in addition $\deg(d_{21})>  \deg(q_2)+k- k_{21}$ and
$\deg(d_{21})> k_1+  k_2- \deg(q_1)-k -k_{21}$, then
    \[
    \dim \kernel T_\Xi \leq \min(\deg(q_1)+k-k_1+\deg(d_{21})+k_{21}-k_2,k_{21}-\deg(q_2)-k+1).
    \]
\end{itemize}
\end{proposition}

\begin{proof}[\bf Proof]
As observed in Proposition \ref{P:2x2kernel1} (iii.b), under the assumptions $k_2> \deg(q_2)+k$ and $k_1< \deg(q_1)+k$, the kernel of $T_\Xi$ has the form
\[
\kernel T_{\Xi} = \left\{\ \sbpm{f_1\\f_2} \colon M \vec{f}_1=0,\, \vec{f}_2=-N\vec{f}_1 \right\},
\]
and thus $\dim \kernel T_\Xi=\dim \kernel M$. Note that $M$ is of size
\[
k_2-k_{21}-\max(deg(q_2)+k-k_{21},0)=k_2-\max(\deg(q_2)+k,k_{21})
\]
times $\deg(q_1)+k-k_1$. Let $L$ be the, possibly vacuous, matrix defined in the proof of Proposition \ref{P:2x2kernel1}. Then $\sbpm{L\\M}$ is a possibly non-square, lower triangular Toeplitz matrix, defined by $d_{21}$, of size $k_2 - k_{21}$ by $\deg(q_1)+k-k_1$. The distinction between cases (i) and (ii)--(iv) is whether $\sbpm{L\\M}$ is ``fat'' or square in case (i) or ``tall'' in cases (ii)--(iv). That is, if $d_{21}(z)=\sum_{j=0}^{\deg(d_{21})-1} \de_j z^j$ and we set $\de_j=0$ for $j\geq \deg(d_{21})$, then in case (i) $\sbpm{L\\M}$ has the form
\[
\pmat{L\\M}=\pmat{\de_0&0&\cdots&0&0&\cdots&0\\\vdots &\ddots&\ddots&\vdots&\vdots&&\vdots\\ \vdots&&\ddots&0&\vdots&&\vdots\\ \de_{k_{2}-k_{21}-1}&\cdots&\cdots& \de_0&0&\cdots&0},
\]
where possibly there are no zero-columns at the right hand-side (in case $k_2 - k_{21}=\deg(q_1)+k-k_1$),
while in cases (ii)--(iv) $\sbpm{L\\M}$ has the form
\begin{equation}\label{LMmat-caseii}
\pmat{L\\M}=\pmat{\de_0&0 &\cdots &0\\ \de_1&\ddots&\ddots&\vdots\\ \vdots&\ddots&\ddots&0\\\de_{\deg(q_1)-k_1-1} &\cdots&\de_1&\de_0\\
\de_{\deg(q_1)-k_1-1} &\cdots&\cdots&\de_1\\ \vdots&&&\vdots\\
\de_{k_2-k_{21}-1}&\cdots&\cdots&\de_{k_1+k_2-\deg(q_1)-k_{21}}}.
\end{equation}

Let us first consider case (i). Set $j_1=\max(\deg(q_2)+k-k_{21},0)$ Then $M$ has the form
\[
M=\pmat{
\de_{j_1}&\cdots&\de_0&0&\cdots&0&0&\cdots&0\\
\vdots&\ddots&&\ddots&\ddots&\vdots&\vdots&&\vdots\\
\vdots&&\ddots&&\ddots&0&\vdots&&\vdots\\
\de_{k_2-k_{21}-1}&\cdots&\cdots&\cdots&\cdots&\de_0&0&\cdots&0\\},
\]
where possibly there are no zero-columns on the right (if $k_2 - k_{21}=\deg(q_1)+k-k_1$) and possibly there are no columns on the left without prescribed zeroes (if $j_1=0$). Write $M$ as
\[
M=\pmat{M_1&M_2&0},
\]
where $M_1$ consists of the first $j_1$ columns of $M$ and $M_2$ consists of the next $k_2-\max(\deg(q_2)+k,k_{21})$ columns of $M$. Then $M_2$ is the square lower triangular Toeplitz operator defined by $d_{21}$ of size $k_2-\max(\deg(q_2)+k,k_{21})$, so that $M_2$ is invertible, since $\de_0\neq 0$. It then follows that
\[
\kernel M=\left\{ \sbpm{v_1\\ -M_2^{-1}M_1 v_1\\ v_3} \colon \mbox{with $v_1\in\BC^{j_1}$ and $v_3\in\BC^{\deg(q_1)+k+k_{21}-k_1-k_2}$}  \right\}.
\]
Therefore, we have
\begin{align*}
\dim \kernel T_\Xi &=\dim \kernel M= j_1+ \deg(q_1)+k+k_{21}-k_1-k_2\\
&= \deg(q_1)+k- k_1 + \max(\deg(q_2)+k-k_{21},0) +k_{21} - k_2,
\end{align*}
as claimed. The formulas for the two special cases (i.a) and (i.b) are obtained by simply filling in the corresponding value for $j_1$.

Now we turn to case (ii). Apart from the fact that $\sbpm{L\\M}$ now has the form \eqref{LMmat-caseii}, we are assuming that $j_1=0$, where $j_1$ is defined as in the first part of the proof. This is precisely the case where $L$ is vacuous (zero rows) so that $M=\sbpm{L\\M}$ is as in \eqref{LMmat-caseii}. Since $\de_0\neq0$ it follows that the columns of $M$ form a linearly independent set of vectors, so that $\kernel M=\{0\}$.

Next we prove the claim of case (iii). In addition to the number $j_1$ defined in the first part of the proof, set $j_2:=k_1+k_2-\deg(q_1)-k-k_{21}$. Note that the assumption in case (i) can be rewritten as $j_2\leq 0$, while case (ii) corresponds to $j_2>0$ and $j_1=0$. Moreover, the restrictions in case (iii) can be reformulated as:
\begin{equation}\label{Case (iii) reform}
j_2>0,\, j_1>0\quad \mbox{and in addition}\quad \deg(d_{21})\leq j_2 \mbox{ or } \deg(d_{21})\leq j_1
\end{equation}
and the subcases as
\[
\mbox{(iii.a):\ $j_1\geq \deg(d_{21}) \geq j_2$,\
(iii.b):\ $j_1\geq \deg(d_{21}) \leq j_2$,\  (iii.c):\ $j_1\leq \deg(d_{21}) \leq j_2$.}
\]
Since we assume $j_1>0$, it follows that $j_1=\deg(q_2)+k-k_{21}$.
We shall first prove the formulas in the special cases and then show that they merge into the general formula for the dimension of the kernel of $T_\Xi$.

Note further that $j_2>0$ implies that the matrix $M$ has the form
\begin{align}\label{M form j2>0}
&\qquad\qquad
M =\pmat{
\de_{j_1}&\cdots& \cdots&\de_{j_3}\\
\vdots&\ddots&\ddots&\vdots\\
\de_{k_2-k_{21}-1}&\cdots& \cdots&\de_{j_2}}\\
 & \mbox{ with $j_3:=j_1+k_1+1-\deg(q_1)-k$, and where $\de_l=0$ for $l<0$,}\notag
\end{align}
where it is not necessarily the case that $j_1>j_2$ as one might interpret from the positions of $\de_{j_1}$ and $\de_{j_2}$. Thus, if $j_3<0$ then there are zeroes in the right upper corner of $M$, but whether there are zeroes in the right upper corner of $M$, or not, is not relevant for the analysis of case (iii). Note that the first condition in \eqref{2x2con} tells us that $\deg(d_{21})\leq k_2-k_{21}-1$. The conditions of case (iii) formulated as in \eqref{Case (iii) reform} tell us that $\deg(d_{21})$ either has to be on the right of $j_1$ or on the right of $j_2$, where this includes the possibility that $\deg(d_{21})$ coincides with $j_1$ or $j_2$. We shall next determine the dimensions of $\kernel M$ in the subcases.\smallskip

\paragraph{\bf Case (iii.a)} In this case $M$ has the form
\[
M=\pmat{0&\cdots&0&\de_{\deg(d_{21})}&\cdots&\cdots&\cdots&\cdots&\de_{j_3}\\ \vdots&&\vdots&0&\ddots&&&&\vdots\\
\vdots&&\vdots&\vdots&\ddots&\ddots&&&\vdots\\ 0&\cdots&0&0&\cdots&0&\de_{\deg(d_{21})}&\cdots&\de_{j_2}}.
\]
Again, and similarly in the other cases to come, we point out that it may happen that $\deg(d_{21})$ coincides with $j_1$, $j_2$ or $k_2-k_{21}-1$, although this is not suggested by the forms of $M$. Note that in this case $M$ has a similar form as in case (i) and a similar argument can be used to determine $\kernel M$. In fact, since we only look for a formula for $\dim \kernel M$ it suffices to note that $M$ has full row rank with the rank given by $\textup{rank}\, M=k_{2}-k_{21}-j_1=k_2-\deg(q_2)-k$ and therefore
\begin{align*}
\dim \kernel M&=\deg(q_1)+k-k_1 -(k_2-\deg(q_2)-k)\\
&=\deg(q_1)-k_1 +\deg(q_2) -k_2 +2k.
\end{align*}
Note that the right hand-side cannot be negative due to our assumptions. Indeed, since $j_2\leq \deg(d_{21})\leq j_1=\deg(q_2)+k-k_{21}$, it follows that
\begin{align*}
\deg(q_1)+k-k_1 & \geq k_1-k_{21}-\deg(d_{21}) \geq k_1 -\deg(q_2)-k.
\end{align*}

\paragraph{\bf Case (iii.b)} Now $M$ looks like
\[
M=\pmat{0&\cdots&0&\de_{\deg(d_{21})}&\cdots&\cdots&\de_{j_3}\\ \vdots&&\vdots&0&\ddots&&\vdots\\
\vdots&&\vdots&\vdots&\ddots&\ddots&\vdots\\ 0&\cdots&0&0&\cdots&0&\de_{\deg(d_{21})}\\
0&\cdots&0&0&\cdots&\cdots&0\\
\vdots&&\vdots&\vdots&&&\vdots\\
0&\cdots&0&0&\cdots&\cdots&0}.
\]
In this case, the dimension of the kernel corresponds to the number of zero-columns, which is the smallest of $j_1-\deg(d_{21})$ and $\deg(q_1)+k-k_1$. Hence
\begin{align*}
\dim \kernel M = \min(\deg(q_1)+k-k_1, \deg(q_2)+k-k_{21}- \deg(d_{21})).
\end{align*}
By assumption of $j_1\geq \deg(d_{21})$, this cannot be a negative number.\smallskip

\paragraph{\bf Case (iii.c)} In the final case $M$ comes out as
\[
M=\pmat{
\de_{j_1} &\cdots&\de_{j_3}\\
\vdots&&\vdots\\
\de_{\deg(d_{21})}&&\vdots\\
0&\ddots&\de_{\deg(d_{21})}\\
\vdots&\ddots&0\\
\vdots&&\vdots\\
0&\cdots&0}
\]
It is clear in this case that the columns of $M$ are linearly independent so that $\dim \kernel M=0$, as claimed.

To complete the proof of statement (iii) it remains to show that the formulas for the three special cases are all subsumed by the general formula. Since $j_1=\deg(q_2)+k-k_{21}$, we have
\[
j_1-j_2=\deg(q_1)+k-k_1 + \deg(q_2)+k-k_2=a_1,
\]
and
\[
a_2=\min(j_1-\deg(d_{21}),\deg(q_1)+k-k_1).
\]
Since $j_1-j_2<\deg(q_1)+k-k_1>0$, we have
\begin{align*}
\dim \kernel T_\Xi &= \max(a_1, a_2, 0)\\
&=\max(j_1-j_2,\min(j_1-\deg(d_{21}),\deg(q_1)+k-k_1),0)\\
&=\min(\max(j_1-j_2,j_1-\deg(d_{21}),0),\deg(q_1)+k-k_1)
\end{align*}
and we need to show that this formula corresponds to the formulas in cases (iii.a)--(iii.c) when restricting to the assumptions made in these cases.\smallskip

\paragraph{\bf Case (iii.a)} If $j_1\geq \deg(d_{21}) \geq j_2$, then $j_1-j_2\geq j_1- \deg(d_{21})\geq 0$ and thus $\dim \kernel T_\Xi=j_1-j_2 = \deg(q_1)+k-k_1 + \deg(q_2)+k-k_2$, using that $j_1-j_2<\deg(q_1)+k-k_1$.\smallskip

\paragraph{\bf Case (iii.b)} If $j_1\geq \deg(d_{21}) \leq j_2$, then we have
\begin{align*}
& \max(j_1-j_2,j_1-\deg(d_{21}),0) = j_1-j_2 + \max(0,j_2-\deg(d_{21}),j_2-j_1)\\
&\qquad = j_1-j_2 + \max(j_2-\deg(d_{21}),j_2-j_1)
=\max(j_1-\deg(d_{21}),0)\\
&\qquad = j_1-\deg(d_{21}) = \deg(q_2)+k - k_{21} - \deg(d_{21}),
\end{align*}
and we obtain
\[
\dim \kernel T_\Xi=\min(\deg(q_2)+k - k_{21} - \deg(d_{21}),\deg(q_1)+k-k_1).
\]

\paragraph{\bf Case (iii.c)} If $j_1\leq \deg(d_{21}) \leq j_2$, then $j_1-j_2$ and $j_1-\deg(d_{21})$ are both non-positive, so that $\dim \kernel T_\Xi=0$.\smallskip

Finally, we turn to the proof of case (iv). Then the assumption translates to $j_1<\deg(d_{21})$ and $j_2<\deg(d_{21})$. In that case $M$ has the form
\[
M=
\pmat{
\de_{j_1}&\cdots& \cdots& \cdots& \cdots&\de_{j_3}\\
\vdots&&&&&\vdots\\
\de_{\deg(d_{21})}&&&&&\vdots\\
0&\ddots&&&&\vdots\\
\vdots&\ddots&\ddots&&&\vdots\\
0&\cdots&0&\de_{\deg(d_{21})}& \cdots&\de_{j_2}}.
\]
It follows that the rank of $M$ is at least equal to the size of the lower triangular block in the left lower corner, which is $k_2-k_{21}-\deg(d_{21})$. However, there is a second lower bound for the rank of $M$, namely the size of that upper triangular block in the right upper corner, in case $j_3<0$. The size of the upper triangular block is equal to $\max(-j_3,0)$. Hence, since $k_2-k_{21}-\deg(d_{21})\geq 0$, it follows that
\begin{align*}
\textup{rank}\, M
&\geq \max(k_2-k_{21}-\deg(d_{21}), \max(-j_3,0))\\
&=\max(k_2-k_{21}-\deg(d_{21}), -j_3)\\
&=\max(k_2-k_{21}-\deg(d_{21}), \deg(q_1)+k-k_1-j_1-1).
\end{align*}
Hence we obtain that
\begin{align*}
&\dim \kernel T_\Xi
 = \dim \kernel M = \deg(q_1)+k-k_1 - \textup{rank}\, M \\
&\quad\leq \deg(q_1)+k-k_1 +\\
&\qquad\qquad - \max(k_2-k_{21}-\deg(d_{21}), \deg(q_1)+k-k_1-j_1-1)\\
&\quad=\min(\deg(q_1)+k-k_1 - (k_2-k_{21}-\deg(d_{21})), j_1+1)\\
&\quad=\min(\deg(d_{21})-j_2,j_1+1)\\
&\quad=\min(\deg(q_1)+k-k_1+\deg(d_{21})+k_{21}-k_2,k_{21}-\deg(q_2)+k+1),
\end{align*}
which proves the upper bound for $\dim \kernel T_\Xi$ in case (iv).
\end{proof}

In Proposition \ref{P:2x2kernel2} there are many cases where the formula does not correspond to the result of Theorem \ref{T:Main} for the case that $\Om$ has a Wiener-Hopf type factorization that is in special form. We next show that the two results are not in conflict, by showing that whenever $\Xi$ in \eqref{2x2} is in special form, the result of Theorem \ref{T:Main} is recovered.

\begin{corollary}\label{C:2x2special}
Let $\Xi$ be given by \eqref{2x2} such that \eqref{2x2con} is satisfied  and $\Xi$ is in special form. Then $\deg(q_1)+k\leq k_1$ or $\deg(q_2)+k\geq k_2$, or if $\deg(q_1)+k> k_1$ and $\deg(q_2)+k< k_2$, then the entries of the matrix $M$ in Proposition \ref{P:2x2kernel1} (iii.b) are all zero (which can only occur in Proposition \ref{P:2x2kernel2} case (iii.b)) so that $\dim \kernel T_{\Xi}=\deg(q_1)+k-1$. In all three cases we obtain that
\[
\dim \kernel T_{\Xi}=\max(\deg(q_1)+k-k_1,0) + \max(\deg(q_2)+k-k_2).
\]
\end{corollary}

\begin{proof}[\bf Proof]
It is only required to prove the statement for the case where $\deg(q_1)+k> k_1$ and $\deg(q_2)+k< k_2$. So assume these two inequalities hold, and let $M$ be the matrix defined in Proposition \ref{P:2x2kernel1} (iii.b). Recall that $\Xi(z)$ is in special form when
\[
\deg(q_1)-k_1\leq \deg(q_2)-k_{21}-\deg(d_{21}).
\]
Thus
\begin{align*}
j_2 & = k_2-k_{21}+k_1-\deg(q_1) -k\\
&\geq k_2-k_{21} -(\deg(q_2)-k_{21}-\deg(d_{21}))-k\\
&= \deg(d_{21}) + k_{2} -\deg(q_2)-k.
\end{align*}
and
\begin{align*}
\deg(q_2)+k-k_{21} &\geq \deg(d_{21})+ \deg(q_1)+k-k_1.
\end{align*}
It is clear that the above inequality remains valid if the left hand-side is replaced with the maximum of $\deg(q_2)+k-k_{21}$ and zero, so that we obtain $j_1\geq \deg(d_{21})+ \deg(q_1)+k-k_1$. In particular, we have $j_2> \deg(d_{21}) \geq 0$ and $j_1> \deg(d_{21}) \geq 0$.

Now we consider the various cases of Proposition \ref{P:2x2kernel2}. The fact that $j_2 >\deg(d_{21})$ and $j_1 > \deg(d_{21})$, excludes the cases (i) ($j_2\leq 0$), (ii) ($j_2>0$, $j_1=0$), (iii.a) ($j_1\geq \deg(d_{21}) \geq j_2>0$), (iii.c) ($j_1\leq \deg(d_{21}) \leq j_2>0$) and (iv) ($j_1< \deg(d_{21}) > j_2$). Hence only case (iii.b) ($j_1\geq \deg(d_{21}) \leq j_2>0$) can occur. The special form condition tells us precisely that $\dim \kernel T_\Xi=\deg(q_1)+k-k_1$, as claimed.
\end{proof}

As can be deduced from Corollary \ref{C:2x2special}, even in the $2 \times 2$ case, there are many cases where the Wiener-Hopf type factorization is not in special form. However, the middle term $\Xi$ of Wiener-Hopf type factorization is not uniquely determined by $\Om$, and so it may well be that there is another Wiener-Hopf type factorization that is in special form. To conclude this section, we consider the case that for $\Om$ we can find a Wiener-Hopf type factorization where all the numerators are monomial (the monomial case), that is, the middle factor $\Xi$ in \eqref{2x2}--\eqref{2x2con} has $d_{21}$ constant, say $d_{21}(z)\equiv 1$. Note that in the case $\deg(d_{21})=0$, both cases in Proposition \ref{P:2x2kernel1} (iii.a) can occur, while from Proposition \ref{P:2x2kernel2} only cases (i), (ii) and (iii.b) can occur.

It turns out that in that case, if the Wiener-Hopf type factorization is in monomial form but not in special form, one can always find a Wiener-Hopf type factorization that is diagonal, and hence in special form.

\begin{lemma}\label{L:2x2mono}
Let $\Om\in\Rat^{2\times 2}$ with poles on $\BT$. Assume that $\Om$ admits a Wiener-Hopf type factorization in monomial form, i.e., with middle term $\Xi$ as in \eqref{2x2}--\eqref{2x2con} such that $\deg(d_{21})=0$. The $\Om$ admits a  Wiener-Hopf type factorization in special form.
\end{lemma}

\begin{proof}[\bf Proof]
Let $\Xi$ as in \eqref{2x2}--\eqref{2x2con} be the middle factor of a monomial Wiener-Hopf type factorization of $\Om$, i.e., with $\deg(d_{21})=0$. Note that the  Wiener-Hopf type factorization is in special form when
\[
k_1-\deg(q_1) \geq k_{21} - \deg(q_2).
\]
If that is the case, then we are done, so we only have to consider the case where $k_1-\deg(q_1) < k_{21} - \deg(q_2)$. Since $q_1$ is a factor of $q_2$, we can write $q_2=q_1 r_1$ for some $r_1\in\cP$ with $\deg(p_1)=\deg(q_2)-\deg(q_1)$. Hence $\deg(p_1)<k_{21}-k_1$. Now set
\[
\Om_-(z)=\pmat{-1&p_1(z) z^{k_1-k_{21}}\\0&1}\ands \Om_+(z)=\pmat{-z^{k_2-k_{21}}&1\\1&0}.
\]
Then $\Om_+(z)$ is a plus function, with plus function inverse, because $k_{21}<k_2$. Moreover, $\Om_-(z)$ is a minus function with minus function inverse, since the non-special form assumption makes the non-constant term a strictly proper rational function with poles only at $0$. We have
\[
z^k\Om_-(z) \Xi(z)\Om_+(z)=\pmat{0& \frac{z^{k_1+k_2-k_{21}}}{q_1(z)}\\ \frac{z^{k_{21}}}{q_2(z)} & \frac{z^{k_2}}{q_2(z)}} \Om_+(z)
=\pmat{\frac{z^{k_1+k_2-k_{21}}}{q_1(z)}& 0\\ 0 & \frac{z^{k_{21}}}{q_2(z)}}.
\]
This shows that $\Om$ also has a Wiener-Hopf type factorization with middle term in diagonal form, which, in particular, is in special form.
\end{proof}

\section{A characterization of $\kernel T_\Xi$ in the $m \times m$ case}\label{S:KernChar}

In this final section we provide a partial extension of Proposition \ref{P:2x2kernel1} to the general case of $m \times m$ matrix functions $\Om$ as well as a recursive procedure to determine the kernel of $T_\Xi$, where $\Xi$ is the middle factor of a Wiener-Hopf type factorization of $\Om$.

\begin{theorem}\label{T:KernalChar}
Let $\Om\in\Rat^{m\times m}$ with poles on $\BT$ and Wiener-Hopf type factorization \eqref{WHfact1} with middle factor $\Xi$ in \eqref{MiddleFact} of the form \eqref{Xi} with conditions \eqref{XiExtra} in place. Then for any $f=(f_1,\ldots,f_m)^T\in H^p_m$ the following are equivalent:
\begin{itemize}
\item[(i)] $f\in \kernel T_\Xi$;

\item[(ii)] $f_1,\ldots,f_m\in\cP$ such that for $j=1,\ldots,m$ we have
\begin{align}\label{mxm-kernelcond}
&\deg(p_{j1}f_1+\cdots + p_{jj}f_{j})
< \deg(q_j) +k.
\end{align}

\item[(iii)]
for $j=1,\ldots,m$, $f_j\in\cP_{l_j-1}$ and $\vec{f}:=(\vec{f}_1,\ldots,\vec{f}_m)^T\in\BC^{l_1+\cdots + l_m}$ is in $\kernel M$, where
\[
l_j=\max(\deg(q_1)-k_1-1,\dots,\deg(q_{j-1})-k_{j-1}-1,\deg(q_j)-k_j,-k)+k
\]
and
\[
M=\pmat{M_{11}&0&\cdots&0\\\vdots&\ddots&\ddots&\vdots\\\vdots&&\ddots&0\\M_{m1}&\cdots&\cdots&M_{mm}}
\]
with for $1\leq i\leq j \leq m$ the matrix $M_{ji}$ being the compression of the Toeplitz operator of $d_{ji}$ by restricting to the first $l_i$ columns and selecting rows
\[
\deg(q_j)+k +1\quad\mbox{to}\quad k_j+l_j,
\]
to be interpreted as vacuous in case $k_j+l_j\leq \deg(q_j)+k$ or $l_i=0$.
\end{itemize}
\end{theorem}

\begin{proof}[\bf Proof]
The structure and arguments of the proof are similar to the proof for the $2 \times 2$ case in Proposition \ref{P:2x2kernel1}. We first prove the equivalence of (i) and (ii). More precisely, we show that in $T_\Xi f$, the $j$-th entry is zero precisely when \eqref{mxm-kernelcond} holds.  From Lemma \ref{L:polykern} we know that $f\in \kernel T_\Xi$ implies that $f\in\cP^m$, hence we may restrict to the case that $f_1,\ldots,f_m$ are polynomials. Since $f\in \Dom(T_\Xi)$, by Lemma \ref{L:MidtermDom}, we know that $(\Xi f)(z)= z^{-k}h(z) + \eta(z)$ for $h=(h_1,\ldots,h_m)^T\in H^p_m$ and $\eta=(r_1/q_1,\ldots,r_m/q_m)^T\in\Rat_0^m(\BT)$, and $h$ and $\eta$ are uniquely determined by $f$ and $\Xi$. Now fix $j$ and consider the $j$-th entry of the identity $(\Xi f)(z)= z^{-k}h(z) + \eta(z)$ multiplied on both sides with $z^kq_j(z)$:
\begin{equation}\label{jentry}
p_{j1}(z)f_1(z)+\cdots + p_{jj}(z)f_{j}(z)=q_j(z) h(z)+ z^k r_j(z).
\end{equation}
Since $\BP (z^{-k}h(z))=0$ if and only if $h\in\cP_{k-1}$, it follows that the $j$-th entry of $T_\Xi f$ is zero if and only if \eqref{jentry} holds with $h\in\cP_{k-1}$. Since $z^k$ and $q_j$ are co-prime, varying $h\in\cP_{k-1}$ and $r_j\in\cP_{\deg(q_j)-1}$, in the right hand side of \eqref{jentry}, any polynomial in $\cP_{\deg(q_j)+k-1}$ can occur. Hence the $j$-th entry of $T_\Xi f$ is zero if and only if
\[
p_{j1}f_1+\cdots+p_{jj}f_{j}\in \cP_{\deg(q_j)+k-1},
\]
that is, if and only if \eqref{mxm-kernelcond} holds. Since $j$ is arbitrary, it follows that (i) and (ii) are equivalent.

Next we prove the equivalence of (ii) and (iii). The restrictions on the degrees of $f_i,\ldots,f_j$ were proved in Lemma \ref{L:polykern}. In terms of Toeplitz operators, condition \eqref{mxm-kernelcond} can be interpreted as saying that
\[
T_{z^{-(\deg(q_j)+k)}}( T_{d_{j1}}f_1+\cdots+ T_{d_{jj}}f_j)=0.
\]
Since $\deg(f_i)<l_i$, in the above identity we can restrict the Toeplitz operator $T_{d_{j1}}$ to its first $l_i$ columns (corresponding to $\cP_{l_i-1}\subset H^p$). In the resulting operator, all rows after $\deg(d_{ji})+ l_i$ are zero rows. Note further that  for all $i>1$ we have $l_i\geq l_{i-1}-1$, so that $l_j\geq l_i-1$ for $1\leq i\leq j$, while also $k_j> \deg(d_{ij})$. As a consequence, we have $\deg(d_{ji})+ l_i\leq k_j+l_j$ for $1\leq i\leq j$. Thus, in the restriction of $T_{d_{j1}}$ to its first $l_i$ columns, all rows after $k_j+l_j$ are zero, for $1\leq i\leq j$, hence we may safely remove all rows of $T_{d_{ji}}$ after the $(k_j+l_j)$-th row. Applying $T_{z^{-(\deg(q_j)+k)}}$ in the above formula corresponds to removing the first $\deg(q_j)+k$ rows, so that the resulting compression of $T_{ji}$ corresponds to $M_{ji}$, We have proved that \eqref{mxm-kernelcond} is equivalent to $\vec{f}$ being in $\kernel \pmat{M_{1j}&\cdots&M_{jj}&0&\cdots&0}$, so that the equivalence of (ii) and (iii) follows.
\end{proof}

In Theorem \ref{T:KernalChar} we did not exploit the special form of the diagonal elements of $\Xi$. Moreover, the lower triangular form of $\Xi(z)$ and resulting block lower triangular form of $M$ in Theorem \ref{T:KernalChar} suggests to determine $\kernel M$ recursively. For this purpose, decompose the matrix $M_{ji}$ for $i<j$ as
\[
M_{ji}=\pmat{N_{ji}\\L_{ji}},
\]
where $N_{ji}$ consist of the first $\max(k_j-\deg(q_j)-k,0)$ rows of $M_{ji}$ and $L_{ji}$ consists of the remaining rows. If $k_j<\deg(q_j)+k$, then $N_{ji}$ is vacuous and $L_{ji}=M_{ji}$. Now define
\[
M_j=\pmat{M_{11}&0&\cdots&0\\\ \vdots&\ddots&\ddots&\vdots\\ \vdots&&\ddots&0\\ M_{j1}&\cdots&\cdots&M_{jj}},\
R_j=\pmat{M_{11}&0&\cdots&0\\\ \vdots&\ddots&\ddots&\vdots\\ \vdots&&\ddots&0\\ M_{(j-1)1}&\cdots&\cdots&M_{(j-1)(j-1)}\\ N_{j1}&\cdots&\cdots&N_{j(j-1)}}.
\]

\begin{proposition}\label{P:mxmkernel}
Let $\Om\in\Rat^{m\times m}$ with poles on $\BT$ and Wiener-Hopf type factorization \eqref{WHfact1} with middle factor $\Xi$ in \eqref{MiddleFact} of the form \eqref{Xi} with conditions \eqref{XiExtra} in place. Then $\kernel M_{1}=\kernel M_{11}=\BC^{\max(\deg(q_1)+k-k_{1},0)}$ and recursively we determine $\kernel M_j$ as follows:\ For $f_i\in\cP_{l_j-1}$, for $i=1,\ldots,j$, we have that $\vec{f}=(\vec{f}_1,\ldots,\vec{f}_j)$ is in $\kernel M_j$ if and only if
\[
(\vec{f}_1,\ldots,\vec{f}_{j-1})\in \kernel R_j=\kernel M_{j-1}\cap \kernel \pmat{N_{j1}&\cdots&N_{j (j-1)}}
\]
and
\begin{equation}\label{fjform}
f_j= r_j+ z^{\max(\deg(q_j)+k-k_j,0)}g_j
\end{equation}
with $r_j\in\cP_{\max(\deg(q_j)+k-k_j,0)-1}$ arbitrarily and $\vec{g}_j=-L_{j1}\vec{f}_{1}-\cdots -L_{j(j-1)}\vec{f}_{j-1}$.
\end{proposition}

\begin{proof}[\bf Proof]
It is clear that for $\vec{f}$ to be in $\kernel M_j$ we need $(\vec{f}_1,\ldots,\vec{f}_{j-1})$ to be in $\kernel M_{j-1}$. Hence we only have to determine the extra conditions imposed by $\vec{f}$ being in the kernel of the last block row of $M_j$. Decomposing $M_{jj}$ in the same way as $M_{ji}$ for $i<j$ we obtain that
\[
M_{jj}=\pmat{0\\ I_{l_j}}\quad \mbox{ if }\quad k_j>\deg(q_j)+k,
\]
or
\[
M_{jj}=\pmat{0& I_{l_j+k_j-\deg(q_j)-k}}\quad \mbox{ if }\quad k_j\leq\deg(q_j)+k.
\]

The statement for $j=1$ is now obvious. For $j>1$ we distinguish between two cases, corresponding to the two forms $M_{jj}$ can have.

First assume that $k_j\leq\deg(q_j)+k$. In that case $N_{ji}$ is vacuous and $L_{ji}=M_{ji}$ for $i<j$, and $M_{jj}$ is surjective. Hence $R_j=M_{j-1}$ so that $\kernel R_j=\kernel M_{j-1}$. Moreover, $\vec{f}$ is in $\kernel M_j$ if $(\vec{f}_1,\ldots,\vec{f}_{j-1})$ is in
$\kernel M_{j-1} =\kernel R_j$ and
\[
M_{jj}\vec{f}_j = - M_{j1}\vec{f}_1-\cdots - M_{j(j-1)}\vec{f}_{j-1}.
\]
Given the form of $M_{jj}$ in this case, the above identity corresponds to $f_j$ being as in \eqref{fjform} since $M_{ji}=L_{ji}$.

Next we consider the case where $k_j>\deg(q_j)+k$. In this case the last block row of $M_j$ has the form
\[
\pmat{N_{j1}&\cdots&N_{j (j-1)}&0\\ L_{j1}&\cdots&L_{j (j-1)}&I_{l_j}}.
\]
Hence for $\vec{f}$ to be in $\kernel M_j$ it is necessary and sufficient for $(\vec{f}_1,\ldots,\vec{f}_{j-1})$ to be in the kernel of  $\pmat{N_{j1}&\cdots&N_{j (j-1)}}$ and
\[
0=L_{j1}\vec{f}_1+\cdots+L_{j(j-1)}\vec{f}_{(j-1)}+\vec{f_j}.
\]
This corresponds to \eqref{fjform} since $\max(\deg(q_j)+k-k_j,0)=0$ and hence $f_j=g_j$, with $g_j$ as in the theorem.
\end{proof}

\begin{remark}
The dichotomy observed in Remark \ref{R:2x2dichotomy} also appears in the general case. Indeed when recursively determining the kernel of $M$ as described in Proposition \ref{P:mxmkernel}, when investigating the restriction imposed by the $j$-th row, in case $k_j\leq \deg(q_j)+k$ no further constraints are added on $f_1,\ldots,f_{j-1}$ and for $f_j$ there is a freedom added of dimension $\deg(q_j)+k-k_j$, whereas in the case where $k_j>\deg(q_j)+k$, $f_j$ is completely determined and a further reduction of the dimension is established by the requirement that $(\vec{f}_1,\ldots,\vec{f}_{j-1})^T$ is in the kernel of $\pmat{N_{j1}&\cdots&N_{j(j-1)}}$.
\end{remark}

\subsection*{Acknowledgements}

This work is based on research supported in part by the National Research Foundation of South Africa (NRF), (Grant Numbers 118513, 127364 and 145688) and the DSI-NRF Centre of Excellence in Mathematical and Statistical Sciences (CoE-MaSS). Any opinion, finding and conclusion or recommendation expressed in this material is that of the authors and the NRF and CoE-MaSS do not accept any liability in this regard.

\end{document}